\documentclass[11pt,reqno]{amsart}
\usepackage{amssymb}
 \usepackage{setspace}
\usepackage{amsmath, amssymb, amsthm, color}
\usepackage{graphics,graphicx,epsfig,psfrag,color}
\usepackage{cite}
\usepackage{bibentry}

\newtheorem{theorem}{Theorem}[section]

\newtheorem{lemma}{Lemma}[section]

\numberwithin{equation}{section}
\newcommand{\D}{\displaystyle}
\newcommand{\DF}[2]{\frac{\D#1}{\D#2}}
\newcommand{\f}{\frac}
\newcommand{\ta}{\theta}

\newcommand{\ppp}{\partial^+ p}
\newcommand{\pmp}{\partial^-p}
\newcommand{\linv}{\lambda^{-1}}
\newcommand{\lp}{\Lambda_+}
\newcommand{\lm}{\Lambda_-}
\newcommand{\vep}{\varepsilon}
\newcommand{\tta}{\theta}

\begin{document}
\title[Sonic curves for 2-D Riemann problems of  nonlinear wave systems]
{The regularity of sonic curves for the two-dimensional Riemann problems of the
 nonlinear wave system of Chaplygin gas}

\author[Q. Wang and K. Song ]{Qin Wang \quad Kyungwoo Song${}^\dag$}

\address{Department of Mathematics, Yunnan University, Kunming, 650091, P.R.China.}
\address{Department of Mathematics and Research Institute for Basic Sciences,
Kyung Hee University, Seoul 130-701,  Korea,}

\thanks{${}^\dag$ Corresponding author}
\email{mathwq@ynu.edu.cn, kyusong@khu.ac.kr}
\date{}
\begin{abstract}
We study the regularity of sonic curves to a two-dimensional Riemann problem for the nonlinear wave system of Chaplygin gas, which is an essential step for the global existence of solutions to the two-dimensional Riemann problems. As a result, we establish the global existence of {\sl uniformly} smooth solutions in the
semi-hyperbolic patches up to the sonic boundary, where the degeneracy of hyperbolicity occurs. Furthermore, we show  the $C^1$ regularity of sonic curves.
\end{abstract}
\maketitle

%%%%%%%%%%%%%%%%%%%%%%%%%%%%%%%%%%%%%%%%%%%%%%%%%%%%%%%%%%%%%%%%%%%%%%%%%%%%%%%%%%%

\section{Introduction}
The purpose of this paper is to provide the regularity of sonic curves arising from a two-dimensional Riemann problem governed by a self-similar nonlinear wave equations.
Let density be $\rho$, velocity $(u,v)$,  the pressure $\tilde p$ be given as a function of $\rho$.
From the well-known 2-D compressible Euler system for isentropic flow,
when the flow is irrotational and the nonlinear velocity terms are ignored,  we can  derive the following nonlinear wave system
\begin{align}
\label{nws}
      \rho_t +  m_x +n_y &=0, \nonumber \\
               m_t + \tilde p_x & =0,  \\
               n_t + \tilde p_y & =0, \nonumber
\end{align}
where $(m,n)=(u\rho, v\rho)$ are momenta.
We refer the readers to \cite{cakeykim2, kim2, kim} for more information on this system.
In this paper we are interested in the 2-D Riemann problem of system \eqref{nws} with Chaplygin gas equation of state $\tilde p=-1/\rho$. The Chaplygin gas, which can be used to depict some dark-energy models in cosmology, has been widely studied. Recently, some interesting and important results have been obtained, especially for the Riemann problem. We refer to \cite{schen, chengyang, lai, serre, wangchenhu} and the references cited therein
for the related results.

For decades, there have been wide and intensive developments in the Riemann problems of conservation laws, in particular the 2-D compressible Euler system and  its simple models \cite{lzy, lzz}.
Such models are the unsteady transonic small disturbance (UTSD) equations \cite{hunter, cakey, cakeykim}, the pressure gradient system \cite{daizhang, ks, lei, sz, yzheng-cpde}, the potential flow \cite{cf1, cf2, elling, kim0}, and the nonlinear wave equations \cite{cakeykim2, kim2, kim, allen} and so on.
For Riemann problems in two-dimensional flow, those governing equations
become quasilinear and mixed types. The type of the flow in the far-field is hyperbolic, while near the origin  the type is mixed. The sonic lines, located between the elliptic part and the hyperbolic part, are important to obtain the global existence of  solutions to the mixed type problem. 

In this paper we consider the presence of the semi-hyperbolic patches with a transonic shock and a sonic curve in a simple wave environment.
The regions identified by the existence of a family of characteristics
that start on sonic curves and end on transonic shock curves are called
semi-hyperbolic patches. The study on this region was initiated in \cite{sz} via the pressure gradient system and is being continued
in \cite{my, qinyuxi, tianyuxi}. These regions are different from the typical fully
hyperbolic regions. As in the result of \cite{gjllzzz}, in several numerical calculations, we
frequently encounter this kind of regions in several configurations in the
two-dimensional Riemann problems of the Euler system as well as its simplified models (Figure 1).
One of the examples for semi-hyperbolic patches happens when air is accelerated
in a planar tunnel over a small inward bulge on one of the walls \cite{courant}.
\begin{figure}
\centering{\includegraphics[clip, width=9cm] {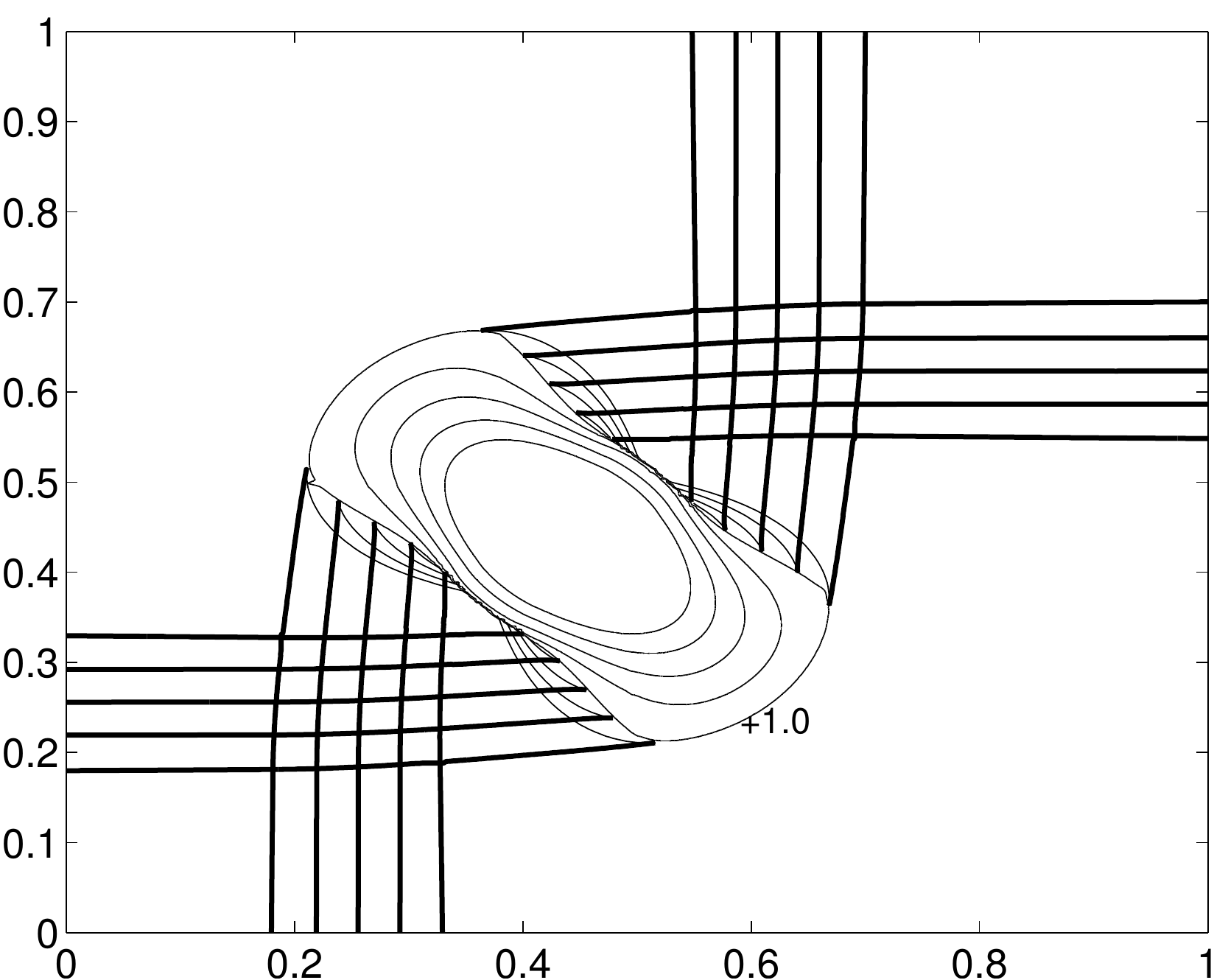} }
\parbox{13cm}{
\caption{\label{Fig1}Four semi-hyperbolic patches,
marked by the short and light curves and overlapped by some bold ones,
in the interaction of two forward and two backward rarefaction waves for
the Euler system with a gas constant $\gamma=1.4$. The five closed
curves at the center are contour curves of the pseudo-Mach number
which mark the subsonic region; bold curves originating from the boundaries
are characteristics; the short and light curves in the semi-hyperbolic
regions are also characteristics. The solution has two axes of symmetry
(Courtesy of Glimm et.~al.~\cite{gjllzzz}).} }
\end{figure}
A global solution to the boundary value problem of the Riemann problem
of a hyperbolic system  can be patched together by pieces along
characteristic lines, sonic curves, shock waves or other natural boundaries.
We believe that it is helpful and essential to understand
these patches of solutions locally before resolving the Riemann problems globally. 

In \cite{huwang}, Hu and Wang have studied  the semi-hyperbolic patches of solutions to the 2-D nonlinear wave system for Chaplygin gas state. They have constructed the global existence of solutions by solving a Goursat-type boundary value problem which has a sonic curve as a degenerate boundary, and verified that the sonic
curve is Lipschitz continuous. However, the smoothness of the sonic curve and the exact behavior of solutions
near the sonic curve  are still open. The main difficulty is  the degeneracy of hyperbolicity near the sonic lines.
In several numerical simulations, we can
see that the demarcation line between hyperbolic regions and elliptic regions is decomposed of sonic curves and
a transonic shock. Also the pressure seems to change smoothly across the sonic curves. Hence, we mainly
focus on the study of the solutions' precise behaviors near sonic curves, and expect better regularity of sonic curves. As a result, we prove that the semi-hyperbolic patches are uniformly smooth up to the sonic curves, and  the sonic curves are $C^1$ continuous.

In the self-similar coordinate $(\xi,\eta)=(x/t,y/t)$, the nonlinear wave system (\ref{nws}) becomes
\begin{align*}
      -\xi \tilde p_\xi - \eta \tilde p_\eta + \tilde p^2 m_\xi +\tilde p^2 n_\eta &=0, \\
       -\xi m_\xi - \eta m_\eta +  \tilde p_\xi  & =0, \\
       -\xi \eta_\xi - \eta n_\eta + \tilde p_\eta        & =0,
  \end{align*}
which yields a second order partial differential equation of $\tilde p$
\begin{align}
\label{eq:one-nonlinear}
  (\tilde p^2-\xi^2)\tilde p_{\xi\xi}-2\xi \eta \tilde p_{\xi\eta} +(\tilde p^2-\eta^2)\tilde p_{\eta \eta}
  + \f{2}{\tilde p}(\xi \tilde p_\xi+\eta \tilde p_\eta)^2
  -2(\xi \tilde p_\xi+\eta \tilde p_\eta)=0.
\end{align}
The two characteristics $\eta=\eta(\xi)$ are defined by
 \begin{align} 
 \label{cha_selfsimilar}
      \f{d\eta}{d\xi}=\f{\xi\eta\pm \sqrt{\tilde p^2(\xi^2+\eta^2-\tilde p^2)}}{\xi^2-\tilde p^2}
 \end{align}
 when $\xi^2+\eta^2>\tilde p^2$.
In the polar coordinate system $(r,\theta)=(\sqrt{\xi^2+\eta^2}, \arctan(\eta/\xi))$,
the equation (\ref{eq:one-nonlinear}) becomes
\begin{align}
\label{no-rta}
  (p^2-r^2)p_{rr}+\f{p^2}{r^2}p_{\ta\ta}+\f{p^2}{r}p_r +\f{2r^2}{p}p_r^2-2rp_r =0,
\end{align}
where $p(r,\theta)$ is a simple notation of  $\tilde p(r\cos\theta,r\sin\theta)$.
Equation (\ref{no-rta}) is elliptic
 in the region of $r^2-p^2<0$, and degenerates when $r+p=0$. In particular, in the area of $r^2-p^2>0$, it is hyperbolic, and
  two characteristics can be defined by
  \[  \f{dr}{d\tta}=\pm \linv =\pm \sqrt{\f{r^2(r^2-p^2)}{p^2}}. \]

In the hyperbolic region, equation (\ref{no-rta}) can be decoupled to
\begin{equation}
\label{decom}
  \partial^+\pmp=Q(\ppp - \pmp)\pmp, \quad \partial^-\ppp=Q(\pmp-\ppp)\ppp,
\end{equation}
where
\[ Q:= \f{r^2}{2p(r^2-p^2)}, \]
and  $\partial^\pm:= \partial_\ta \pm \linv \partial_r$ are  directional derivatives along the positive
and negative characteristics, respectively.
That is, we have a new system
 \begin{eqnarray}
  \label{matrix-eq}
   \left( \begin{array}{c}
     \ppp \\ \pmp \\ p
     \end{array} \right)_\theta
     +
   \left( \begin{array}{ccc}
        -\lambda^{-1} & 0 & 0 \\
        0 & \lambda^{-1} & 0  \\
        0 & 0 & 0
     \end{array} \right)
   \left( \begin{array}{c}
      \ppp \\ \pmp \\ p
     \end{array} \right)_r
    = \left( \begin{array}{c}
       Q(\pmp-\ppp)\ppp \\ Q(\ppp-\pmp)\pmp \\ \frac{1}{2}(\ppp+\pmp)
        \end{array} \right).
 \end{eqnarray}
Using this decomposition, the existence of solutions to (\ref{matrix-eq}) can be obtained in the hyperbolic region
even if we cannot get the Riemann invariants \cite{huwang, my, sz}.
Letting 
\[ R:=\ppp, \quad S:=\pmp, \]
 we have
 \[ R_\ta -\lambda^{-1}R_r = Q(S-R)R, \quad S_\ta +\lambda^{-1}S_r = Q(R-S)S. \]

 \begin{figure}
\centering{\includegraphics[clip, width=13cm]{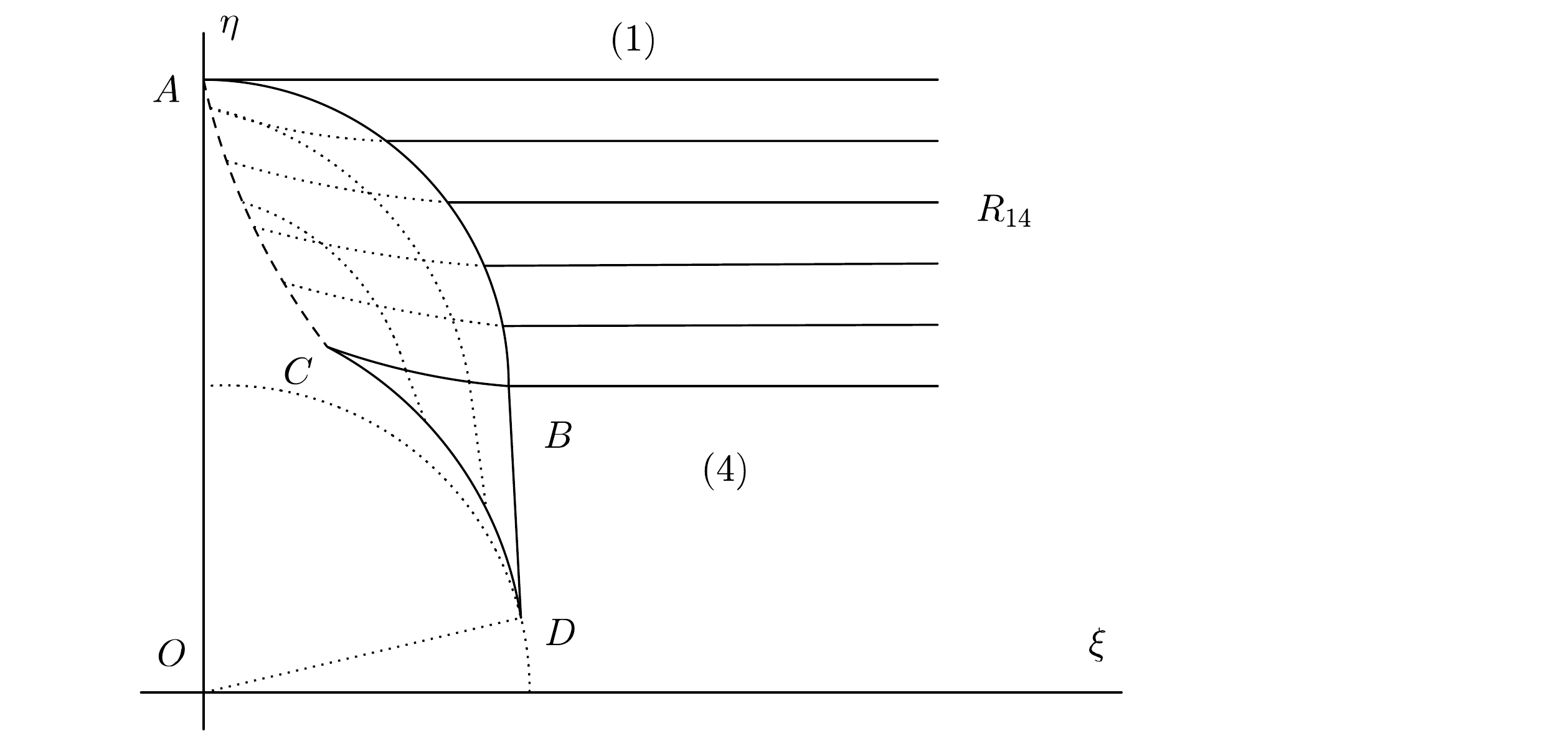}}
\parbox{14cm}{\caption{\label{Fig 2}
{{\bf The semi-hyperbolic region $ABC$}}.  Here $(1)$ and $(4)$ denote the constant states
of $p=p_1$ and $p=p_4$, respectively, and $R_{14}=-\eta$ is a planar wave. 
Arc $CD$ is an envelope and the region of $OACD$ is
a subsonic region. We are interested in the solutions in a curvilinear triangle $ABC$ and the regularity of
the sonic
curve $AC$ for the further research in the whole region of $OABD$.}  }
\end{figure}
At this point, let us give a short description of the problem and known results.  In the self-similar coordinates,
two constant states $(p_1, m_1, n_1)$ and $(p_4, m_4, n_4)$  are given with $p_1<p_4<0$, and a planar wave
\begin{equation*}
   R_{14}:  \left\{\begin{array}{ll}
       p(\xi, \eta)=-\eta, \\ [5pt]
      n(\xi, \eta)=\ln \f{p_4}{p(\xi, \eta)}, \\[5pt]
      m(\xi, \eta)=m_1=m_4=0,
   \end{array}\right.
\end{equation*}
connects these two constant states in the domain $\{\xi>0, \eta>0\}$. Let the point $(0, -p_1)$ be $A$, which is a sonic point  satisfying  $\xi^2+\eta^2=p^2$. We
draw a positive characteristic curve  defined by (\ref{cha_selfsimilar}) passing $A$ in the planar wave, and
intersecting the bottom boundary of the planar wave at $B$. Next we draw a negative characteristic
defined by (\ref{cha_selfsimilar}) from $B$ toward the subsonic region, and let $C$ be the ending point of the negative characteristic on the sonic curve.  
Upon this configuration in Figure 2, if we assume that the negative characteristic $BC$ is given, 
then the region $ABC$ bounded
by two characteristics $AB$ and $BC$ forms a 
semi-hyperbolic region, and  $AC$ is a sonic curve.
Here, the curve $BC$ serves as a spine to support the whole patch.

Let $W:=(R, S, p)$, and we also denote the arcs $AB$ and $BC$ by $\Gamma_+$ and $\Gamma_-$, respectively. For given boundary conditions on $\Gamma_\pm$, the global existence  and a priori $C^1$ estimate of $W$ in the domain
$ABC$ have been
obtained in \cite{huwang} as follows:
\begin{theorem}(\textbf{Global existence} \cite{huwang}).
 \label{knowntheorem}
   Assume that two constant states $(p_1, m_1, n_1)$ and $(p_4, m_4, n_4)$  are given with $p_1<p_4<0$ and a planar
   wave connects these two constant states in the domain $\xi>0$.
   For a given negative characteristic curve $\Gamma_-$, there exists a classical solution $W$ to the Goursat
   boundary value problem of (\ref{matrix-eq}) in the whole domain $ABC$ with a Lipschitz continuous sonic boundary $AC$
   if $p_4$ is close enough to $p_1$. Furthermore,  $\partial^+ p$ and $\partial^- p$ are strictly negative if the boundary data of
$\partial^+ p, \partial^- p$ on $\Gamma_\pm$ are strictly negative, and $\partial^+ p= \partial^- p$ on the sonic curve. And
$\partial^+ p, \partial^- p$ are uniformly bounded in the whole semi-hyperbolic region $ABC$.
\end{theorem}

In fact, we do not know exactly how far we can extend the $C^1$ solution in Theorem \ref{knowntheorem} 
towards the sonic boundary. 
 In this paper, we improve the results of \cite{huwang} in the sense that the solutions in the semi-hyperbolic
 patches are {\sl uniformly} smooth up to the sonic curve, where the degeneracy of hyperbolicity occurs, and that the sonic curve is $C^1$ continuous. In fact, the smoothness of sonic curve is a necessary requirement for the
 global existence of solutions in the region of $OABD$ in Figure 2 for the future research.
 The main theorem is as follows.

 \begin{theorem}(\textbf{Regularity}) \label{mainthm}
    For the Goursat problem with prescribed characteristic boundaries $\Gamma_+$ and $\Gamma_-$, there exists a
    global uniformly smooth solution  of (\ref{matrix-eq}) in the domain $ABC$ under the same conditions as
   in Theorem \ref{knowntheorem}. Moreover, the sonic line $AC$ is $C^1$ continuous, and  the two derivatives $\partial^+p$ and  $\partial^-p$ approach a common value on the sonic curve with a rate of
   $O(\sqrt{r+p})$.
 \end{theorem}

 To overcome the difficulty of degeneracy, we will employ a new different coordinate system  $(r,t)$ with
$ t=\sqrt{r+p(r,\ta)}$ in Section 2. Under the new coordinates, we derive new systems for $(R, S)$.
The proof of the main theorem is based heavily on the uniform bounds of $t^\delta R_r$ and $t^\delta S_r$
for any $1<\delta<2$.
In the following section, we will prove the above main theorem using the bootstrap method to get the uniform
estimates as in \cite{lei, qinyuxi}.  Especially 
upon the method of \cite{qinyuxi}, we show that a global smooth solution in \cite{huwang} is extended  up to the 
sonic boundary by obtaining the convergence rate of $\partial^+p$ and $\partial^-p$ near the sonic curve. 
With those estimates and the convergence rate, we are able to prove $\partial^+p$ and $\partial^-p$ are uniformly
continuous on the sonic curve, and furthermore, the sonic curve is $C^1$ continuous.

For recent related progresses on the classical solution near the sonic lines for 2-D the pressure gradient systems and Euler systems, we  would like to mention the work of Zhang and Zheng \cite{tianyuxi} first. In \cite{tianyuxi}, for any given smooth curve as a sonic line with any specified non-tangential derivative of
the speed of sound, they have constructed classical solutions from the sonic line towards the
supersonic side by establishing a convergent iteration scheme. On the contrary, in \cite{qinyuxi} Q. Wang and Y. Zheng solved a Goursat problem and considered the regularity of solutions from the supersonic domain
towards the sonic curve. They considered the problem from a different aspect but also hoped it
can be used to build a transonic solution that bridges between a semi-hyperbolic patch to
an elliptic domain. For the two-dimensional Euler system, M. Li and Y. Zheng \cite{my}
studied the semi-hyperbolic patches of solutions. In a forthcoming paper \cite{swz}, we will consider the regularity of sonic curves for the Euler system.

\section{Main Results}
\subsection{Derivation of New Systems}
We introduce the new coordinates $(r,t)$, where
\[ t=\sqrt{r+p(r,\ta)}\]
and $r$ is a radius-component in the polar coordinates.
Then $\{t=0\}$ flattens the sonic boundary, and
 \[ Q= \f{-r^2}{2t^2(r-t^2)(2r-t^2)}, \quad \linv=\f{rt\sqrt{2r-t^2}}{r-t^2}. \]
We note that $\linv$ goes to zero as $t \to 0$.

The sonic line is defined by the limit of the following level curves
$\{(r,\theta): \; r+p(r,\theta)= \vep \}$.
Let such level curves be $\theta=\theta_\vep(r)$. Then
\[ \theta'_\vep (r)= - \f{p_r + 1}{p_\theta}. \]
Here, $p_\theta$ is negative and bounded in the whole domain. However,
\[ p_r = \f{R-S}{2\linv}, \]
where $R-S$ and $\linv$ vanish simultaneously as $t \to 0$. It seems necessary to know the rate of vanishing
of these two components, more precisely the vanishing rate of $R-S$ with respect to $t$, in order to know the regularity of the sonic line.

In the coordinate system $(r, t)$, we have
\begin{equation}
 \label{eq:RS}
   \left\{\begin{array}{ll}
     \displaystyle  R_t -  \f{2t\linv}{S-\linv}R_r = \f{2t}{S-\linv} Q(S-R)R, \\ [5pt]
     \displaystyle  S_t +  \f{2t\linv}{R+\linv}S_r = \f{2t}{R+\linv} Q(R-S)S.
   \end{array}\right.
\end{equation}
Since $\linv$ goes to zero as $t \to 0$,
 we note that $S-\linv \ne 0$ and $R+\linv \ne 0$ in the $(r, t)$-plane when $t$ is sufficiently
small. Letting 
\[ \lp:=\f{2t \linv}{R+\linv}, \quad \lm:=-\f{2t \linv}{S-\linv},   \]
 we have
\begin{align*}
   \left(\f{1}{R}\right)_t & =\f{2t^2Q}{1-S^{-1}\linv}\left(\f{1}{S}-\f{1}{R}\right)\f{1}{t}
      +\f{2t\linv S^{-1}}{1-S^{-1}\linv}\left(\f{1}{R}\right)_r, \\
  \left(\f{1}{S}\right)_t & =\f{2t^2Q}{1+R^{-1}\linv}\left(\f{1}{R}-\f{1}{S}\right)\f{1}{t}
      -\f{2t\linv R^{-1}}{1+R^{-1}\linv}\left(\f{1}{S}\right)_r.
\end{align*}
Defining
\[  U:= \f{1}{R}+\f{1}{S}, \quad V:= \f{1}{S}- \f{1}{R}, \]
where $V=0$ on the sonic line,
 we have
\begin{align*}
   U_t = \f{2tQUV \linv}{(1-S^{-1}\linv)(1+R^{-1}\linv)}
            + \f{2t\linv S^{-1}}{1-S^{-1}\linv}\left(\f{1}{R}\right)_r
            - \f{2t \linv R^{-1}}{1+R^{-1}\linv}\left(\f{1}{S}\right)_r,
\end{align*}
and
\begin{align}
\label{eq:V}
    V_t = \f{2tQ V(-2+ \linv V)}{(1-S^{-1}\linv)(1+R^{-1}\linv)}
            - \f{2t\linv S^{-1}}{1-S^{-1}\linv}\left(\f{1}{R}\right)_r
            - \f{2t \linv R^{-1}}{1+R^{-1}\linv}\left(\f{1}{S}\right)_r.
\end{align}
Furthermore, we define
\[ G:=\partial_+ R - \partial_- R, \quad H:=\partial_+ S - \partial_- S,  \]
where $\partial_\pm:= \partial_t \pm \Lambda_\pm \partial_r$.
Then
\begin{align*}
   \partial_-G &= \f{\partial_-\lp - \partial_+\lm}{\lp - \lm}G
     + (\partial_+\partial_-R - \partial_-\partial_-R), \\
   \partial_+ H &= \f{\partial_-\lp - \partial_+\lm}{\lp - \lm}H
     + (\partial_+\partial_+ S - \partial_-\partial_+ S).
\end{align*}
Here,
\[ \lp -\lm = 2t \linv \f{R+S}{(R+\linv)(S-\linv)},   \]
and by direct calculations,
\begin{align*}
   \f{\partial_-\lp - \partial_+\lm}{\lp - \lm} = \f{2}{t} + h(r,t),
\end{align*}
where
\begin{align*}
h(r,t) := & \f{t(3r-t^{2})}{(r-t^2)(2r-t^2)}
+ \f{ r^3(3R-3S+4\linv) + 2t\linv (t^5+r^2t-3rt^3)}
 {(R+\linv)(S-\linv)(r-t^2)^2\sqrt{2r-t^2}}.
\end{align*}
Note that  $h \to 0$ as $t \to 0$.
Furthermore,
\begin{align}
\label{fi}
  \partial_+\partial_-R - \partial_-\partial_-R
 = (\lp-\lm)(\partial_- R)_r
 :=tf_1(r,t)R_r+ tf_2(r,t)S_r + t^2 f_3(r,t),
\end{align}
where
\begin{align*}
  f_1(r,t)& :=\f{r^2(2R-S)}{(S-\linv)(r-t^2)(2r-t^2)}E, \\
  f_2(r,t)&:=\f{-R}{S-\linv}\Big(1+\f{R-S}{S-\linv}\Big)\f{r^2}{(r-t^2)(2r-t^2)}E, \\
  f_3(r,t)&:= \f{r R(R-S)}{(S-\linv)(r-t^2)^2 (2r-t^2)^{3/2}}
  \Big\{ \f{-3r^2t^2+rt^4+r^3}{(S-\linv)(r-t^2)}
    - \f{t(3r-2t^2)}{\sqrt{2r-t^2}}   \Big\}E,
\end{align*}
and
  \[ E(r,t) := 2r\f{\sqrt{2r-t^2}}{r-t^2}\Big(\f{1}{R+\linv}+\f{1}{S-\linv}\Big). \]
We note that functions $E, f_j$ for $j=1, 2, 3$ are all bounded in bounded regions.
Also we can obtain by direct calculations that
\begin{align}
\label{gi}
\partial_+\partial_+ S - \partial_-\partial_+ S
 :=t g_1(r,t) R_r+ t g_2(r,t) S_r + t^2 g_3(r,t),
\end{align}
where
\begin{align*}
  g_1(r,t)& := \f{-S}{R+ \linv}\Big(1+\f{S-R}{R+\linv}\Big)\f{r^2}{(r-t^2)(2r-t^2)}E, \\
  g_2(r,t)&:=\f{r^2(2S-R)}{(R+ \linv)(r-t^2)(2r-t^2)}E, \\
  g_3(r,t)&:=\f{r S(S-R)}{(R+ \linv)(r-t^2)^2 (2r-t^2)^{3/2}}
  \Big\{ \f{3r^2t^2 - rt^4 - r^3}{(R+ \linv)(r-t^2)}
    - \f{t(3r-2t^2)}{\sqrt{2r-t^2}}    \Big\}E.
\end{align*}
We note that functions $g_j$ for $j=1, 2, 3$ are all bounded in bounded regions, too.

\subsection{Regularity of Solutions near Sonic Curves}
Since we are concerned with the regularity of semi-hyperbolic patches near the sonic line, for any fixed point $(\overline{r},0)$ on the sonic line $AC$, we shall take a new point $B^{\prime}(\overline{r},t_{B^{\prime}})$, where $t_{B^{\prime}}$ is positive and small
so that $B^{\prime}$ remains in the domain $ABC$ (Figure 3). Then, through the point $B^{\prime}$ we can draw the positive and negative characteristic curves $r_{+}(B^{\prime})$ and $r_{-}(B^{\prime})$ until
they intersect the sonic line $AC$ in two points $A^{\prime}$ and $C^{\prime}$, respectively.
Let us denote the region $ABC$
by $\Omega$ and  $A^{\prime}B^{\prime}C^{\prime}$
by $\mathcal{D}$.

Since $R, S$ are uniformly bounded and negative in the domain $\Omega$ and $R=S$ on the sonic line, we can obtain
$h(r,t), f_{j}(r,t),g_{j}(r,t)(j=1,2,3)$ are also uniformly bounded in the small subdomain $\mathcal{D}$.
Let
\begin{align}
\label{k}
K_{1} &=\max\limits_{(r,t)\in \mathcal{D}}\Big\{|h(r,t)|,\, |f_{3}(r,t)|,\,
|g_{3}(r,t)|\Big\}, \\
K_{2} &=\max\limits_{(r,t)\in \mathcal{D}}\Big\{|f_{j}(r,t)|,\, |g_{j}(r,t)|,\; j=1,2\Big\},
\end{align}
and
   $K_{3}=\min\limits_{(r,t)\in \mathcal{D}} |E|$.
We can see that $K_{1}$ tends to zero and $K_{2},K_{3}$ have positive bounds as $t_{B^{\prime}}\rightarrow 0^+$. If $t_{B^{\prime}}=0$, the domain $\mathcal{D}$ degenerates into a point on the sonic line, and two constants $K_2$ and $K_3$ satisfy that $2K_2=|E(\overline{r},0)|=K_3$
with $K_3> 2K_2 \delta^{-1}$ for any $\delta\in (1,2)$. Thus, we can let
 $K_{3}> 2 \delta^{-1} K_{2} e^{2K_{1}t_{B^{\prime}}}$
and $K_{1}<K_{2}$ by taking suitable small $t_{B^{\prime}}$.

Let $\Omega(\overline{r},0)$ denote a bounded domain surrounded by a positive characteristic and a negative characteristic starting from $(\overline{r},0)$ and the characteristics $B^{\prime}C^{\prime}$ and $B^{\prime}A^{\prime}$.
For any $(r,t)\in\Omega(\overline{r},0),$ let $a$ and $b$ be intersections of the negative characteristic and the positive characteristic  through $(r,t)$ with the boundaries $B^{\prime}C^{\prime}$ and $B^{\prime}A^{\prime}$, respectively.
Let
\begin{equation*}
M_{0}=\max\Big\{\max\limits_{\Omega(\overline{r},0)}\DF{|H(t_{a})|}{K_{2}t_{a}^{2-\delta}e^{K_{1}t_{a}}}+1,  \quad\max\limits_{\Omega(\overline{r},0)}\DF{|G(t_{b})|}{K_{2}t_{b}^{2-\delta}e^{K_{1}t_{b}}}+1\Big\},
\end{equation*}
where $t_a$ and $t_b$ are $t$-coordinates of points $a$ and $b$, respectively.
 As discussed before, $M_{0}$ is well-defined and uniform
in the domain $\Omega(\overline{r},0)$, but depending on $\delta$.
\begin{figure}
\centering{\includegraphics[clip, width=15cm]{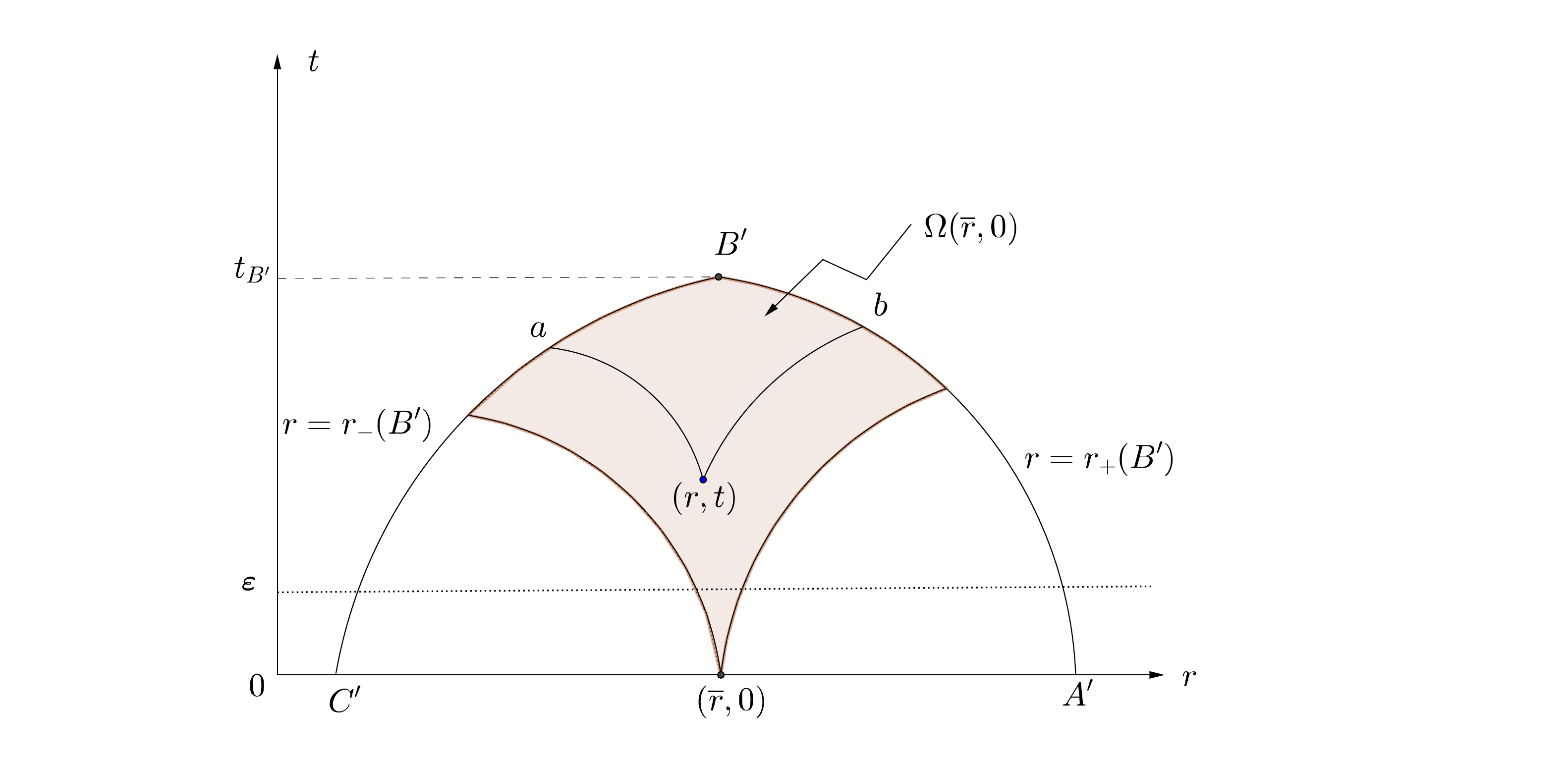}}
\parbox{14cm}{\caption { The regions of $A'B'C'$ and $\Omega(\overline{r},0)$}{\label{Fig 3}}}
\end{figure}
We can prove the following result.
\begin{lemma}
\label{lemma:3.1}
For any $(r,t)\in\Omega(\overline{r},0)$, and any $\delta\in(1,2)$, there hold
\begin{equation}
\label{eq:3.5}
 |t^{\delta}R_{r}|\leq M, \quad |t^{\delta}S_{r}|\leq M,
\end{equation}
where $M>0$ is depending on $\delta$. That is, $t^{\delta}R_{r}$ and $t^{\delta}S_{r}$ are
uniformly bounded in the domain $\Omega(\overline{r},0)$.
\end{lemma}
\begin{proof}
For a fixed $\varepsilon\in(0,t_{B^{\prime}}),$ let
\begin{equation*}
\Omega_{\varepsilon}:=\Big\{(r,t)|\,\varepsilon\leq t\leq t_{B^{\prime}}, \,r_{-}(B^{\prime})\leq r \leq r_{+}(B^{\prime})\Big\}
\,\cap\,\Omega(\overline{r},0).
\end{equation*}
Defining $M_\vep=\max\limits_{\Omega_{\varepsilon}}\{|t^{\delta}R_{r}|,|t^{\delta}S_{r}|\},$ we have
\begin{equation}
\label{eq:3.4}
t^{\delta}|R_{r}|\leq M_\vep,\quad t^{\delta}|S_{r}|\leq M_\vep,
\end{equation}
in $\Omega_{\varepsilon}$ with $M_\vep$ depending on $\varepsilon$. If for any $\varepsilon\in(0,t_{B^{\prime}}),$ we have
$M_\vep\leq M_{0},$
then \eqref{eq:3.4}
hold in the whole domain $\Omega(\overline{r},0)$.
If not, then there exists  $\varepsilon_{0}\in(0,t_{B^{\prime}})$ making $M_{\vep_0} > M_{0}$.
So we let
$$M_{\vep_0}=\max\limits_{\Omega_{\varepsilon_{0}}}\{|t^{\delta}R_{r}|,|t^{\delta}S_{r}|\}.$$

For $(r_{\varepsilon_{0}},\varepsilon_{0})\in\Omega(\overline{r},0)$, we denote the two characteristics passing through
$(r_{\varepsilon_{0}},\varepsilon_{0})$ by $r_{+}^{a}$ and $r^{b}_{-}$, where $a,b$ are the  intersection points with $B^{\prime}C^{\prime}$
and $B^{\prime}A^{\prime}$ respectively. Here $a=a_{\varepsilon_{0}}$ and $b=b_{\varepsilon_{0}}$ are dependent on $\varepsilon_{0}$ and for convenience we omit the subscript
$\varepsilon_{0}$.
From the equation
\begin{align*}
   \partial_-G &= \f{\partial_-\lp - \partial_+\lm}{\lp - \lm}G
     + (\partial_+\partial_-R - \partial_-\partial_-R),
\end{align*}
 we have
\begin{align*}
\partial_{-}\Big(G(r,t) &\exp\int_{t_{b}}^{t}-\Big(\frac{2}{\tau}+h(r_{-}(\tau),\tau)\Big)d\tau\Big)\\
& =(\partial_{+}\partial_{-}R-\partial_{-}\partial_{-}R)\exp\int_{t_{b}}^{t}
-\Big(\frac{2}{\tau}+h(r_{-}(\tau),\tau)\Big)d\tau.
\end{align*}
Integrating this along the negative characteristic passing through $(r_{\varepsilon_{0}},\varepsilon_{0})$,  we get
\begin{align*}
  G(r,t)\left(\frac{t_{b}}{t}\right)^{2} & \exp\int_{t_{b}}^{t}\Big(-h(r_{-}(\tau),\tau)\Big)d\tau\\
& =G(r,t_{b})-\int_{t}^{t_{b}}(\partial_{+}\partial_{-}R-\partial_{-}\partial_{-}R)
\left(\frac{t_{b}}{\tau}\right)^{2}
\exp\int_{t_{b}}^{\tau}\Big(-h(r_{-}(s),s)ds \Big)\; d\tau.
\end{align*}
Thus, from \eqref{fi} and \eqref{k} we obtain
\begin{align*}
 |G(r,t)|(\frac{t_{b}}{t})^{2}&\exp\int_{t_{b}}^{t}\Big(-h(r_{-}(\tau),\tau)\Big)d\tau\nonumber\\
&\leq |G(r,t_{b})|+\int_{t}^{t_{b}}\left(\frac{t_{b}}{\tau}\right)^{2}\exp(K_{1}(t_{b}-\tau))
|\partial_{+}\partial_{-}R-\partial_{-}\partial_{-}R |d\tau\nonumber\\
&\leq |G(r,t_{b})|+t_{b}^{2}e^{K_{1}t_{b}}\int_{t}^{t_{b}}
\frac{1}{\tau^{2}}\Big(2K_{2}M_{\vep_0}\tau^{1-\delta}+K_{1}\tau^2\Big)d\tau\nonumber\\
&= |G(r,t_{b})|+t_{b}^{2}e^{K_{1}t_{b}}\Big[\frac{2K_{2}M_{\vep_0}}{\delta}(t^{-\delta}-t_{b}^{-\delta})
+K_{1}(t_{b}-t)\Big]\nonumber\\
&< \frac{2K_{2}M_{\vep_0}t_{b}^{2}e^{K_{1}t_{b}}}{\delta t^{\delta}},
\end{align*}
if $|G(r,t_{b})|+K_1(t_b-t)t_b^2 e^{K_1t_b}-2\delta^{-1}K_2M_{\vep_0}e^{K_1t_b}t_b^{2-\delta}<0$, which is true by
 the definition of $M_0$ and the fact of $M_{\vep_0}>M_0$.
Thus we obtain
\begin{align*}
|G(r,t)|&< t^{2-\delta}\frac{2K_{2}M_{\vep_0} e^{K_{1}t_{b}}}{\delta}\exp\int_{t_{b}}^{t}h(r_{-}(\tau ),\tau)d\tau\nonumber\\
&\leq 2 \delta^{-1} t^{2-\delta} K_{2}M_{\vep_0} e^{2K_{1}t_{b}}.
\end{align*}
On the other hand,
\begin{equation*}
G(r,t)=(\lp-\lm)R_r=2t\linv \Big(\f{1}{R+\linv}+\f{1}{S-\linv}\Big) R_{r}(r,t).
\end{equation*}
Thus we obtain
\begin{align*}
|R_{r}|(t=\vep_0)
&<\frac{1}{t^{\delta}}\cdot\frac{K_{2}M_{\vep_0}e^{2K_{1}t_{b}}}
{\delta\cdot| \f{1}{R+\linv}+\f{1}{S-\linv}|} \cdot \f{r-t^2}{r\sqrt{2r-t^2}} \nonumber\\
&\leq \frac{1}{t^{\delta}}\cdot\frac{2K_{2}M_{\vep_0}e^{2K_{1}t_{b}}}{\delta K_{3}}
\leq \frac{M_{\vep_0}}{t^{\delta}}.
\end{align*}
So we have a strict inequality such that
\begin{eqnarray}
\label{eq:3.7}
 |R_r| < M_{\vep_0}/ \varepsilon_0^\delta
\end{eqnarray}
on the line segment $t=\varepsilon_{0}$.
Similarly, integration of the equation
\[  \partial_+ H = \f{\partial_-\lp - \partial_+\lm}{\lp - \lm}H
     + (\partial_+\partial_+ S - \partial_-\partial_+ S) \]
  along the positive characteristic  makes
\begin{equation*}
 |H(r, t)|<t^{2-\delta}\frac{2K_{2}M_{\vep_0}e^{2K_{1}t_{a}}}{\delta}.
\end{equation*}
Also by the definition of $H$, we get the following estimate
\begin{equation}
\label{eq:3.8}
|S_{r}|_{t=\varepsilon_{0}}< M_{\vep_0} / \varepsilon_0^{\delta}.
\end{equation}
 According to \eqref{eq:3.7} and \eqref{eq:3.8}, $|t^\delta R_r|$ and $|t^\delta S_r|$ do not have
the maximum values on the line segment $t=\varepsilon_{0}$, that is, the maximum happens on  $\varepsilon_{0}<t\leq t_{B'}$. Hence we conclude that
this argument also holds in a larger domain $\Omega_{\varepsilon^{\prime}}$ with $\varepsilon^{\prime}<\varepsilon_{0}$.
Repeating this process, we can extend the domain larger and larger
 until it reaches the whole domain $\Omega(\overline{r},0)$. Thus, we find a constant $M$ depending on $\delta$ only and satisfying \eqref{eq:3.5}. The proof  is completed. \end{proof}

%%%%%%%%%%%%%%%%%%%%%%%%%%%%%%%%%%%%%%%%%%%%%%%%%%%%%%%%%%%%%%%%%%

Next, for any $(\overline{r},0)\in A 'C ',$ we can take a small open interval
$(\overline{r} - r_0,\overline{r} + r_0)\subset A^{\prime}C^{\prime}$.
Suppose the negative and positive characteristics
passing through points $P(\overline{r} - r_0,0)$ and
$Q(\overline{r} + r_0,0)$ intersect
the boundaries $B^{\prime}C^{\prime}$ and $B^{\prime}A^{\prime}$
at $P^{\prime}$ and $Q^{\prime}$, respectively.
Then we have the same estimates in the region of $B^{\prime}P^{\prime}PQQ^{\prime}$ in the following lemma.
The proof of it is very similar to that of the previous lemma, so let us omit the proof.

\begin{lemma}
\label{lm:3.2}
For any $\delta \in(1,2)$, there exists a constant $M$ only depending on the interval $(\overline{r}-r_0,\overline{r}+r_0)$ and $\delta$ such that
\begin{equation*}
  |t^{\delta}R_{r}|\leq M, \quad |t^{\delta}S_{r}|\leq M
\end{equation*}
hold in the domain $B^{\prime}P^{\prime}PQQ^{\prime}$.
\end{lemma}

Now we can show that $|V(r,t)|/t$ is uniformly bounded in the whole domain.
\begin{lemma}
\label{le3}
$|V(r,t)|/t$ is uniformly bounded in the domain $ABC$.
\end{lemma}
\begin{proof}
Let $(\overline{r},0)\in AC$. From \eqref{eq:V} we have
\begin{align}
\label{eq:Vl}
    V_t &= \f{2tQ V(-2+ \linv V)}{(1-S^{-1}\linv)(1+R^{-1}\linv)}
            - \f{2t\linv S^{-1}}{1-S^{-1}\linv}\left(\f{1}{R}\right)_r
            - \f{2t \linv R^{-1}}{1+R^{-1}\linv}\left(\f{1}{S}\right)_r \nonumber \\
        &= \f{r^2(2-\linv V)}{(r-t^2)(2r-t^2)(1-S^{-1}\linv)(1+R^{-1}\linv)}\cdot \f{V}{t} \nonumber \\
        & \qquad\quad + t^{2-\delta} \f{2r\sqrt{2r-t^2}}{(r-t^2)RS}\Big[\f{t^\delta R_r}{(1-S^{-1}\linv)R}
        +   \f{t^\delta S_r}{(1+R^{-1}\linv)S}  \Big] \nonumber \\
&=:l_{1}(r,t)\frac{V}{t}+l_{2}(r,t)t^{2-\delta}.
\end{align}
It is obvious that $\lim\limits_{t\rightarrow 0^+}l_{1}(r,t)=1$, and
$l_{2}(r,t)$ is bounded in the domain $\mathcal{D}$.  Moreover, we note that
\begin{equation*}
\f{l_{1}(r,t)-1}{t}= \f{3rt^2-t^4+ r^2\linv V  -3rt^2\linv V +\linv t^4 V +
      (2r^2 -3rt^2 +t^4)\lambda^{-2}R^{-1}S^{-1}}{t(1-S^{-1}\linv)(1+R^{-1}\linv)(r-t^2)(2r-t^2)}
\end{equation*}
also tends to zero as $t\rightarrow 0^+$ since $V \to 0$ as $t \to 0^+$. Let
$$K_{4}=\max_{\mathcal{D}}\Big|\frac{l_{1}(r,t)-1}{t}\Big|, \quad
K_{5}=\max_{\mathcal{D}}\; |l_{2}(r,t)|.$$

\noindent
Now we rewrite \eqref{eq:Vl} in the form of
\begin{equation*}
\partial_{t}\Big(V \exp\Big(\int_{t}^{t_{B^{\prime}}}\f{l_{1}(r,\tau)}{\tau}d\tau\Big)\Big)
= l_{2}(r,t)t^{2-\delta} \exp\Big(\int_{t}^{t_{B^{\prime}}}\f{l_{1}(r,\tau)}{\tau}d\tau\Big),
\end{equation*}
that is,
\begin{align}
\label{eq:V-diff}
 \partial_{t}\Big(\f{V}{t}\exp\Big(\int_{t}^{t_{B^{\prime}}}\f{l_{1}(r,\tau)-1}{\tau}d\tau\Big)\Big)
=l_{2}(r,t)t^{1-\delta}\exp\Big(\int_{t}^{t_{B^{\prime}}}\f{l_{1}(r,\tau)-1}{\tau}d\tau\Big).
\end{align}
Integration of the above equation from $t$ to $t_{B^{\prime}}$ yields that
\begin{align*}
\frac{V}{t}(r, t_{B^{\prime}})-\frac{V}{t}(r, t)\exp\Big(\int_{t}^{t_{B^{\prime}}}\f{l_{1}(r,\tau)-1}{\tau}\Big)
=\int_{t}^{t_{B^{\prime}}} l_{2}(r,\tau)\tau^{1-\delta}
\exp\Big(\int_{\tau}^{t_{B^{\prime}}}\f{l_{1}(r,s)-1}{s}ds\Big) \; d\tau.
\end{align*}
Thus
\begin{align*}
 e^{-K_{4}t_{B^{\prime}}}\Big|\frac{V}{t}(r,t)\Big|
& < \Big|\frac{V}{t}(r, t)\Big| \exp\Big(\int_{t}^{t_{B^{\prime}}}\frac{l_{1}(r,\tau)-1}{\tau}d\tau\Big)\\
&\leq \Big|\frac{V}{t}\Big|(r,t_{B^{\prime}})
+\int_{t}^{t_{B^{\prime}}}|l_{2}(r,\tau)\tau^{1-\delta}|
\exp\Big(\int_{\tau}^{t_{B^{\prime}}}\Big|\frac{l_{1}(r,s)-1}{s}\Big|ds\Big) \; d\tau\\
&\leq  \Big|\frac{V}{t}\Big|(r, t_{B^{\prime}})+e^{K_{4}t_{B^{\prime}}}
\f{K_{5}t_{B^{\prime}}^{2-\delta}}{2-\delta},
\end{align*}
that is,
\begin{equation}
\label{eq:3.17}
 |\frac{V}{t}(r, t)|<e^{K_{4}t_{B^{\prime}}}\left\{\max_{(r, t_{B^{\prime}})\in ABC}|\frac{V}{t}|(r, t_{B^{\prime}})+
 e^{K_{4}t_{B^{\prime}}}\f{K_{5}t_{B^{\prime}}^{2-\delta}}{2-\delta}\right\}:=\widehat{M}.
\end{equation}
Because of the arbitrariness of $t$ and the continuity of $V$, we obtain that $|V(r,t)|\leq \widehat{M}t$ holds for any $t\in[0,t_{B^{\prime}}]$.
\end{proof}
Thus we see that the two derivatives $\partial^+p$ and    $\partial^-p$ approach a common value on the sonic curve with a rate of
   $O(\sqrt{r+p})$.
Furthermore, we have the following result.
\begin{lemma}
 $R, S$, and $V(r, t)/t$ is uniformly continuous in the domain $ABC$ including the sonic line.
\end{lemma}
\begin{proof}
 Let us choose two sonic points $P_1(r_1,0), P_2(r_2,0) \in AC$.
 Then we draw a positive characteristic $\gamma_+$ passing through $P_2$ and a negative characteristic $\gamma_-$ passing through $P_1$. Let
  $Q(r_b,t_b) \in \Omega$ be an intersection point of $\gamma_+$ and $\gamma_-$.  We note that $Q$ approaches a sonic point as $|r_2-r_1| \to 0$. From (\ref{eq:RS}), we see that
 \begin{equation}
 \label{eq:partialRS}
   \partial_-R =\f{2t^2}{S-\linv}QR \cdot \f{S-R}{t}, \quad
    \partial_+S =\f{2t^2}{R+ \linv}QS \cdot \f{R-S}{t}.
 \end{equation}
Since $V(r,t)/t$ is uniformly bounded, we have
 $|\partial_-R |,\; |\partial_+ S| \leq L$ for some constant $L$.

 Now integrating the first equation of (\ref{eq:partialRS}) along $\gamma_-$ from $P_1$ to
 $Q$ gives
 \[  |R(r_b,t_b) - R(r_1, 0)| \leq \int^{t_b}_0  \Big| \f{2t^2QR}{S-\linv}\cdot \f{S-R}{t}\Big|dt
    \leq Lt_b. \]
Similary the second equation of (\ref{eq:partialRS}) along $\gamma_+$ from $P_2$ to
 $Q$ gives $$|S(r_b,t_b)-S(r_2,0)| \leq L t_b. $$
Moreover, since $R=S$ on $AC$ and $R, S$ are continuous on $\Omega$,
$R(r_j,0)=S(r_j,0)$ for $j=1,2$, and $|R(r_b,t_b)-S(r_2,0)| \to 0$ as $t_b \to 0$.
Then
\begin{align*}
    |R(r_1,0)-R(r_2,0)| & \leq |R(r_1, 0)- R(r_b,t_b)| + |S(r_2, 0)-S(r_b,t_b)|
                           + |R(r_b,t_b)-S(r_b,t_b)|  \\
                        & \leq w(t_b),
\end{align*}
where $w(t_b):=2Lt_b + |R(r_b,t_b)-S(r_b,t_b)|$.
 Similarly,
 $|S(r_1,0)-S(r_2,0)| \leq w(t_b)$.
Since $w(t_b) \to 0$ as $|r_2 - r_1| \to 0$,
which implies that $R$ and $S$ are both continuous on the sonic line $AC$,
we prove that $R$ and $S$ are both uniformly continuous in the whole domain $ABC$.

Let $\vep>0$ be given and $t_b \in (0, t_{B'})$ with $t_{B'} \ll 1$ will be determined later.
 From (\ref{eq:V-diff}), we have
 \begin{align*}
 \partial_{t}\Big(\DF{V}{t}\exp\Big(\int_{t}^{t_{b}}\DF{l_{1}(r,\tau)-1}{\tau}d\tau\Big)\Big)
=l_{2}(r,t)t^{1-\delta}\exp\Big(\int_{t}^{t_b}\DF{l_{1}(r,\tau)-1}{\tau}d\tau\Big),
\end{align*}
 which yields
 \begin{align*}
    \f{V}{t}(r, t_b) - \f{V}{t}(r,0)\exp\Big(\int^{t_b}_0 \f{l_1 -1}{\tau} d\tau \Big)
    = \int^{t_b}_0 l_{2}(r,\tau) \tau^{1-\delta}\exp\Big(\int_{\tau}^{t_b}\DF{l_{1}(r,s)-1}{s}ds\Big)d\tau.
 \end{align*}
Then
\begin{align*}
   \Big| \f{V}{t}(r, t_b) -\f{V}{t}(r,0) \Big| & \leq
   \Big| \f{V}{t}(r,0) \Big|  \Big( \exp\int^{t_b}_0 \Big| \f{l_1 - 1}{\tau} \Big| d\tau -1 \Big)
    + \int^{t_b}_0 |l_2 \tau^{1-\alpha}| \exp \Big( \int^{t_b}_\tau \f{l_1 - 1}{s}ds \Big) \\
     & \leq  2 \widehat{M} K_4 t_b + \f{K_5 t_b^{2-\delta} e^{K_4 t_b}}{2-\delta} \\
     & \leq K_0 t^{2-\delta}_b
\end{align*}
 for some constant $K_0$. Now we take sufficiently small $t_b$ so that $K_0 t^{2-\delta}_b < \vep/4$.
 For this fixed $t_b$, since $V/t$ is continuous inside $\Omega$,  we can take $d>0$ such that, if $|r_2 - r_1|\leq d$, then we have
 \[ \Big| \f{V}{t}(r_1, t_b) - \f{V}{t}(r_2, t_b) \Big| < \f{\vep}{4}.  \]
Thus for any $|r_1-r_2|\leq d$, we have
 \begin{align*}
    \Big| \f{V}{t}(r_1, 0) - \f{V}{t}(r_2, 0) \Big| & \leq
     \Big| \f{V}{t}(r_1, 0) - \f{V}{t}(r_1, t_b) \Big|
       + \Big| \f{V}{t}(r_1, t_b) - \f{V}{t}(r_2, t_b) \Big| \\
           & \qquad \qquad  + \Big| \f{V}{t}(r_2, t_b) - \f{V}{t}(r_2, 0) \Big|   \\
           &  <   2 K_0 t^{2-\delta}_b   + \f{\vep}{4} < \vep.
 \end{align*}
 Therefore, $V/t$ is uniformly continuous in the whole domain.
\end{proof}

{\bf Proof of the Main Theorem: }
We consider the sonic curve as the limit of the level curves $\ta=\ta_\vep(r)$ in the $(r,\ta)$ plane:
\[r+p(r,\ta)=\vep.\]
In the semi-hyperbolic regions, $R, S$ are both strictly negative and uniformly bounded, and thus
\[p_\ta=\frac{R+S}{2}\]
is negative and bounded. Therefore, each $\ta_\vep^{\prime}(r)=-(1+p_r)/p_\ta$ is well defined. Next we have
\[
 p_{r}=\frac{R-S}{2\lambda^{-1}}=\frac{R-S}{t}\cdot\frac{r-t^{2}}{2r\sqrt{2r-t^2}}.
 \]
According to Lemma \ref{le3}, $p_{r}$ is uniformly bounded. Therefore $|\ta_\vep^{\prime}(r)|$ is  also uniformly bounded in the whole domain, including the sonic line $AC$. Moreover, since $p_r$ and $p_\ta$ are uniformly continuous
in the whole domain, $\ta_\vep^{\prime}(r)$ is uniformly continuous as well. Hence $\ta'(r)$ is continuous on the sonic line. That is, the sonic line
is $C^1$ continuous.

%%%%%%%%%%%%%%%%%%%%%%%%%%%%%%%%%%%%%%%%%%%%%%%
\subsection*{Acknowledgements}
Qin Wang's research was supported in part
by Chinese National Natural Science Foundation under grant 11401517, and
Yunnan University Foundation 2013CG021 and W4030002.

\end{document}